\documentclass[11pt]{amsart}
\usepackage{amsmath,mathrsfs,cite}

\usepackage{graphicx}
\usepackage[small]{caption}
\usepackage[colorlinks=true,urlcolor=blue,breaklinks=true,
citecolor=red,linkcolor=blue,linktocpage,pdfpagelabels,bookmarksnumbered,bookmarksopen]{hyperref}
\usepackage[english]{babel}
\usepackage[utf8]{inputenc}
\usepackage{enumerate}
\usepackage[hyperpageref]{backref}
\usepackage[left=2.8cm,right=2.8cm,top=2.8cm,bottom=2.8cm]{geometry}

\newcommand{\R}{{\mathbb R}}

\newcommand{\E}{{\mathcal E}}

\newcommand {\Div} {\mbox{\rm div}} 

\numberwithin{equation}{section}
\newtheorem{theorem}{Theorem}[section]
\newtheorem{proposition}[theorem]{Proposition}
\newtheorem{lemma}[theorem]{Lemma}

\newtheorem{remark}[theorem]{Remark}

\newtheorem{definition}[theorem]{Definition}

\theoremstyle{definition}

\newcommand{\bgset}[1]{\big\{#1\big\}}

\newcommand{\norm}[2][]{\left\|#2\right\|_{#1}}

\DeclareMathOperator{\dvg}{div}

\title[Stability of eigenvalues for variable exponent problems]{Stability of eigenvalues for \\ variable exponent problems}

\author{Francesca Colasuonno}
\author{Marco Squassina}

\address{Dipartimento di Informatica \newline
Universit\`a degli Studi di Verona
\newline\indent
C\'a Vignal 2, Strada Le Grazie 15,
I-37134 Verona, Italy}
\email{marco.squassina@univr.it}

\address{Istituto per le Applicazioni del Calcolo ``M. Picone''\newline
Consiglio Nazionale delle Ricerche\newline\indent
Via dei Taurini 19,
00185 - Roma, Italy}
\email{fracolasuonno@gmail.com}

\thanks{The present research  was partially supported by
 Gruppo Nazionale per l'Analisi Matematica,
 la Probabilit\`a e le loro Applicazioni (INdAM)}

\begin{document}
	



\begin{abstract}
In the framework of variable exponent Sobolev spaces,
we prove that the variational eigenvalues defined by inf sup procedures of Rayleigh ratios for
the Luxemburg norms are all stable under uniform convergence of the exponents.
\end{abstract}

\maketitle

\section{Introduction and main result}

\noindent
The differential equations and variational problems involving $p(x)$-growth conditions arise
from nonlinear elasticity theory and electrorheological fluids, and have been the target of various investigations, especially in regularity theory and in nonlocal problems 
(see e.g.\ \cite{acerbming,acerbming2,acp,cp,ruz,Base} and the references therein).
Let $\Omega\subset\R^N$, with $N\ge2$, be a bounded domain with Lipschitz boundary and let $p:\bar\Omega\to\R^+$ be a continuous function such that
\begin{equation}\label{p-+}
1<p_-:=\inf_{\Omega}p \leq p(x) \leq \sup_{\Omega}p =: p_+ < N\,\,\quad\mbox{for all $x\in\Omega$}.
\end{equation} 
We also assume that $p$ is log-H\"older continuous, namely
\begin{equation}\label{logholder}|p(x)-p(y)|\le -\frac{L}{\log|x-y|}\end{equation}
for some $L>0$ and for all $x,y\in\Omega$, with $0<|x-y|\le1/2$. 
From now, we denote by 
$$
\mathscr C:=\big\{p\in C(\bar\Omega) \,:\, p\mbox{ satisfies \eqref{p-+} and \eqref{logholder}}\big\}
$$
the set of admissible variable exponents.
The goal of this paper is to study the stability of the (variational) eigenvalues 
with respect to (uniform) variations of $p$ for the problem 
\begin{equation}
\label{ep}
-\mathrm{div}\Big(p(x)\left|\frac{\nabla u}{K(u)}\right|^{p(x)-2}\frac{\nabla u}{K(u)}\Big)=\lambda S(u) p(x)\left|\frac u{k(u)}\right|^{p(x)-2}\frac u{k(u)},\mbox\quad u\in W^{1,p(x)}_0(\Omega),
\end{equation}
where we have set 
$$
K(u):=\|\nabla u\|_{p(x)},\quad k(u):=\|u\|_{p(x)},\quad 
S(u):=\frac{\displaystyle\int_\Omega p(x)\left|\dfrac{\nabla u}{K(u)}\right|^{p(x)}dx}{
\displaystyle\int_\Omega p(x)\left|\dfrac u{k(u)}\right|^{p(x)}dx}.
$$
Following the argument contained in  \cite[Section 3]{Franzi}, it is possible to derive equation \eqref{ep} 
as the Euler-Lagrange equation corresponding to the minimization of the Rayleigh ratio 
\begin{equation}
\label{rayleigh}
\frac{K(u)}{k(u)}=\frac{\|\nabla u\|_{p(x)}}{\|u\|_{p(x)}},\quad\mbox{among all }u\in W^{1,p(x)}_0(\Omega)\setminus\{0\},
\end{equation}
where $\|\cdot\|_{p(x)}$ denotes the Luxemburg norm of the variable exponent Lebesgue space $L^{p(x)}(\Omega)$ (see
Section~\ref{recalls}).  
This minimization problem has been firstly introduced in
\cite{Franzi} as an appropriate replacement for the
{\em inhomogeneous} minimization problem
$$
\frac{\displaystyle\int_{\Omega}|\nabla u|^{p(x)} dx}{\displaystyle\int_{\Omega}|u|^{p(x)} dx},\quad\mbox{among all }u\in W^{1,p(x)}_0(\Omega)\setminus\{0\},
$$
which was previously considered in \cite{sticinesi} to define the first eigenvalue $\lambda_1$ of the $p(x)$-Laplacian.
In \cite{sticinesi}, sufficient conditions for $\lambda_1$ defined in this way to be zero or positive are provided. 
In particular, if $p(\cdot)$ has a strict local minimum (or maximum) in $\Omega$, then $\lambda_1=0$.
Arguing as in \cite[Lemma A.1]{Franzi}, it can be shown that the functionals $k$
and $K$ are differentiable with 
\begin{align*}
\langle K'(u),v\rangle &=\frac{\displaystyle\int_\Omega p(x)\left|\dfrac{\nabla u}{K(u)}\right|^{p(x)-2}\dfrac{\nabla u}{K(u)}\cdot\nabla v \,dx}{
\displaystyle\int_\Omega p(x)\left|\dfrac{\nabla u}{K(u)}\right|^{p(x)}dx} \quad\text{for all $u,v\in W^{1,p(x)}_0(\Omega)$,}  \\
\langle k'(u),v\rangle &=\frac{\displaystyle\int_\Omega p(x)\left|\dfrac{u}{k(u)}\right|^{p(x)-2}\dfrac{u}{k(u)}v\, dx}{\displaystyle\int_\Omega p(x)\left|\dfrac{u}{k(u)}\right|^{p(x)}dx} \quad\text{for all $u,v\in W^{1,p(x)}_0(\Omega)$}.
\end{align*}
Therefore, all critical values of \eqref{rayleigh} are eigenvalues of \eqref{ep} and vice versa.
The $m$th eigenvalue $\lambda^{(m)}_{p(x)}$ of \eqref{ep} can be obtained as 
$$
\lambda^{(m)}_{p(x)}:=\inf_{K\in\mathcal W^{(m)}_{p(x)}}\sup_{u\in K} \|\nabla u\|_{p(x)},
$$
where $\mathcal W^{(m)}_{p(x)}$ is the set of symmetric, compact subsets of 
$\{u\in W^{1,p(x)}_0(\Omega):\|u\|_{p(x)}=1\}$ 
such that $i(K)\ge m$, and $i$ denotes the Krasnosel'ski\u{\i} genus. 
In \cite{Franzi} existence and properties of the first eigenfunction were studied, 
while in \cite{MR3122341} a numerical method to compute the first eigenpair of \eqref{ep} 
was obtained and 
the symmetry breaking phenomena with respect to the 
constant case were observed.
The growth rate of this sequence of eigenvalues 
was investigated in \cite{persqu}, getting a natural replacement for the growth 
estimate for the case $p$ constant (cf.\ \cite{MR1017063,peral}),
$$
\lambda^{(m)}_p\sim m^{N/p},\quad\,\,\,
\lambda^{(m)}_{p}:=\inf_{K\in\mathcal W^{(m)}_{p}}\sup_{u\in K} \|\nabla u\|_{p}^p,
$$ 
where $\mathcal W^{(m)}_{p}$ is the set of symmetric, compact subsets of 
$\big\{u\in W^{1,p}_0(\Omega):\|u\|_{p}=1\big\}$.

\vskip2pt
\noindent
In this paper we focus on the right continuity of the maps
$$
\E_m:(C(\Omega),\|\cdot\|_\infty)\to\R,\quad\,\,
\E_m(p(\cdot)):=\lambda^{(m)}_{p(x)},\quad\,\, m\geq 1.
$$
We set
$$
\mathscr S:=\Big\{(p_h)\subset \mathscr C:\,
\text{$\exists\, \bar h\geq 1$ s.t.\ $\sup_{h\geq \bar h} \sup_\Omega p_{h}<p_I^*$,\,\,\, $p_I^*:=\frac{Np_I}{N-p_I},\,\,\, p_I:=\inf_{h\geq \bar h}\inf_\Omega p_{h}$
\Big\}}.
$$
We say that $\E_m$ is {\em right-continuous} if 
$$
\E_m(p_h(\cdot))\to \E_m(p(\cdot)),\quad \text{as $h\to\infty$},
$$
whenever $p\in\mathscr C,\,(p_h)\subset \mathscr S$, $p_h\to p$ uniformly in $\Omega$ and $p(x)\le p_h(x)$ for all $h\in\mathbb N$ and $x\in\Omega$.
\vskip3pt
\noindent
We have the following main result.

\begin{theorem}
\label{main}
$\E_m$ is right-continuous for all $m\ge 1$.
\end{theorem}

\begin{remark}\rm
For a constant $p\in (1,N)$, problem \eqref{ep} 
reduces to the well-known eigenvalue problem
for the $p$-Laplacian operator (see e.g.\ \cite{LindAMS,Lindbook}), namely
$$
-{\rm div}(|\nabla u|^{p-2}\nabla u)=\lambda |u|^{p-2}u,\quad u\in W^{1,p}_0(\Omega).
$$
In this particular case the continuity of variational eigenvalue has been investigated in \cite{huang,champion_depascale2007,
parini2011,littig_schuricht2014,stab,stab2}
and, more recently, in \cite{DM14} in presence of a weight function and including the case where the domain $\Omega$ is unbounded. 
With exception of \cite{stab,stab2,huang}, all these
contributions tackle the problem by studying the $\Gamma$-convergence of the norm functionals.
\end{remark}

\begin{remark}\rm
As pointed out by Lindqvist \cite[see Section 7]{stab}, already in the constant case, the convergence from {\em below} of the $(p_h)$ to $p$ does {\em not} guarantee the convergence of the eigenvalues, unless the domain $\Omega$ is sufficiently smooth.
\end{remark}

\begin{remark}\rm 
The same result holds replacing the Krasnosel'ski\u{\i} genus
with a general index $\operatorname{i}$ with the following properties:
\begin{itemize}
\item[(i)]
$\operatorname{i}(K)$ is an integer greater or equal than $1$ and is defined
whenever $K$ is a nonempty, compact and symmetric subset of a
topological vector space such that $0\not\in K$;
\item[(ii)]
if $X$ is a topological vector space and
$K\subseteq X\setminus\{0\}$ is compact, symmetric and nonempty,
then there exists an open subset $U$ of $X\setminus\{0\}$
such that $K\subseteq U$ and
$\operatorname{i}(\widehat{K}) \leq \operatorname{i}(K)$
for any compact, symmetric and nonempty $\widehat{K}\subseteq U$\,;
\item[(iii)]
if $X, Y$ are two topological vector spaces,
$K\subseteq X\setminus\{0\}$ is compact, symmetric and nonempty
and $\pi:K\to Y\setminus\{0\}$ is continuous and
odd, we have $\operatorname{i}(\pi(K)) \geq \operatorname{i}(K)$\,.
\end{itemize}
Examples are the Krasnosel'ski\u{\i} genus and
the $\mathbb{Z}_2$-cohomo\-logical index
\cite{fadell_rabinowitz1977, fadell_rabinowitz1978}.
\end{remark}

\medskip
\section{Preliminary results}
\label{recalls}

\noindent
The variable exponent Lebesgue space $L^{p(x)}(\Omega)$ consists of all measurable functions $u:\Omega\to\mathbb R$ having $\varrho_{p(x)}(u)<\infty$, where 
$$\varrho_{p(x)}(u):=\int_{\Omega} | u(x) |^{p(x)} dx$$
is the $p(x)$-modular. $L^{p(x)}(\Omega)$ is endowed with the Luxemburg norm $\|\cdot\|_{p(x)}$ defined by
\begin{equation*} 
\label{Luxdef}
\lVert u \rVert_{p(x)}:= 
\inf \Big\{ \gamma > 0 : \varrho_{p(x)}(u/\gamma) \leq 1 \Big\}.
\end{equation*}
The norm $\|u\|_{p(x)}$ is in close relation with the $p(x)$-modular $\varrho_{p(x)}(u)$, as shown for instance by unit ball property \cite[Theorem 1.3]{cinese} which we report here for completeness.

\begin{proposition}\label{ubp}
Let $p\in L^\infty(\Omega)$ with $1<p_-\le p_+<\infty$. Then, for all $u\in L^{p(x)}(\Omega)$ the following equivalence holds
$$\|u\|_{p(x)}<1 \,(=1;\,>1)\quad\Longleftrightarrow\quad \varrho_{p(x)}(u)<1 \,(=1;\,>1).$$
\end{proposition}

The variable exponent Sobolev space $W^{1,p(x)}(\Omega)$ consists of all $L^{p(x)}(\Omega)$-functions 
having distributional gradient $\nabla u\in L^{p(x)}(\Omega)$, and is 
endowed with the norm 
$$
\|u\|_{1,p(x)}=\|u\|_{p(x)}+\|\nabla u\|_{p(x)}.
$$
Under the smoothness assumption \eqref{logholder}, we denote by $W^{1,p(x)}_0(\Omega)$ the closure of 
$C^\infty_0(\Omega)$ with respect to the norm $\|u\|_{1,p(x)}$ and we endow 
$W^{1,p(x)}_0(\Omega)$ with the equivalent norm  $\|\nabla u\|_{p(x)}$. 
For further details on the variable exponent Lebesgue and Sobolev spaces, we refer the reader to \cite{Base}.
\smallskip
\noindent
We now recall from \cite{dalmaso1993} the notion of $\Gamma$-convergence that will be useful in the sequel.

\begin{definition}\rm 
Let $X$ be a metrizable topological space and let $(f_h)$ be a sequence of functions from $X$ to $\overline{\mathbb{R}}$. The \emph{$\Gamma$-lower limit} and the \emph{$\Gamma$-upper limit} of the sequence $(f_h)$ are the functions from $X$ to $\overline{\mathbb{R}}$ defined by 
\begin{gather*}
\Big(\Gamma-\liminf_{h\to\infty} f_h\Big)(u)
= \sup_{U\in \mathcal{N}(u)}
\Big[\liminf_{h\to\infty}
\bigl(\inf\{f_h(v):v\in U\}
\bigr)\Big]\,,\\
\Big(\Gamma-\limsup_{h\to\infty} f_h\Big)(u)
=\sup_{U\in \mathcal{N}(u)}
\Big[\limsup_{h\to\infty}
\bigl(\inf\{f_h(v):v\in U\}\bigr)\Big]\,,
\end{gather*}
where $\mathcal{N}(u)$ denotes the family of all open neighborhoods of $u$ in $X$.
If there exists a function $f:X\to\overline{\R}$ such that 
$$
\Gamma-\liminf_{h\to\infty}f_h=\Gamma-\limsup_{h\to\infty}f_h=f,
$$
then we write $\Gamma-\lim\limits_{h\to\infty} f_h=f$
and we say that $(f_h)$ \emph{$\Gamma$-converges} to its \emph{$\Gamma$-limit} $f$.
\end{definition}

\noindent
For any $p\in\mathscr C$, we define $\mathscr E_{p(x)}:L^1(\Omega)\to[0,\infty]$ 
as
\begin{equation}\label{Epx}
\mathscr E_{p(x)}(u):=\begin{cases}\|\nabla u\|_{p(x)}\quad&\mbox{if }u\in W^{1,p(x)}_0(\Omega),\\
+\infty\quad&\mbox{otherwise}\end{cases}
\end{equation}  
and $g_{p(x)}:L^1(\Omega)\to[0,\infty)$  as
$$g_{p(x)}(u):=\begin{cases}\|u\|_{p(x)}\quad&\mbox{if }u\in L^{p(x)}(\Omega),\\ 
0\quad&\mbox{otherwise.}\end{cases}$$

\begin{proposition}\label{p51} The following properties hold:
\begin{itemize}
\item[(a)] $g_{p(x)}$ is even and positively homogeneous of degree $1$;
\item[(b)] for every $b\in\mathbb R$ the restriction of $g_{p(x)}$ to $\{u\in L^1(\Omega)\,:\, \mathscr E_{p(x)}(u)\le b\}$ is continuous. 
\end{itemize}
\end{proposition}
\begin{proof} (a) follows easily from the definition of $g_{p(x)}$.
(b) Let $(u_n)\subset \{u\in L^1(\Omega)\,:\, \mathscr E_{p(x)}(u)\le b\}$ converge to $u$ in $L^1(\Omega)$ and consider a subsequence $(u_{h_n})$. By the definition of $\mathscr E_{p(x)}$, we know that $(u_{h_n})$ is bounded in $W^{1,p(x)}_0(\Omega)$ which is reflexive, hence there exists a subsequence  $(u_{h_{n_j}})$ that converges weakly to $\bar u$ in $W^{1,p(x)}_0(\Omega)$. Since $W^{1,p(x)}_0(\Omega)$ is compactly embedded in $L^{p(x)}(\Omega)$, cf. \cite[Proposition 2.2 and Lemma 5.5]{diening}, $u_{h_{n_j}}$ converges strongly to $\bar u$ in $L^{p(x)}(\Omega)$. By the arbitrariness of the subsequence $(u_{h_n})$, we get that the whole sequence $u_n\to\bar u$ in  $L^{p(x)}(\Omega)$ and also in $L^1(\Omega)$. Therefore, $u=\bar u$ and the proof is concluded. 
\end{proof}

\begin{lemma}\label{conv} Let $p,\,(p_h)\subset \mathscr C$ be such that $p_h\to p$ pointwise. Then, for all $w\in C^1_\mathrm{c}(\Omega)$
$$\lim_{h\to\infty}\|\nabla w\|_{p_h(x)}=\|\nabla w\|_{p(x)}.$$ 
\end{lemma}
\begin{proof} By means of \cite[Corollary 3.5.4]{Base}, we know that 
$$
\|\nabla w\|_{p(x)}\le \liminf_{h\to\infty}\|\nabla w\|_{p_h(x)}. 
$$
It remains to prove that 
$$
\|\nabla w\|_{p(x)}\ge \limsup_{h\to\infty}\|\nabla w\|_{p_h(x)}. 
$$
If $\nabla w=0$ in $\Omega$, the conclusion in obvious, so we can assume that $\|\nabla w\|_{p(x)}>0$. By hypothesis, $p_h\to p$ pointwise and $p_h(x)<N$ for all $h\in\mathbb N$ and $x\in\Omega$, hence
$$
\begin{gathered}
\left|\frac{\alpha\nabla w}{\|\nabla w\|_{p(x)}}\right|^{p_h(x)}\to\left|\frac{\alpha\nabla w}{\|\nabla w\|_{p(x)}}\right|^{p(x)}\quad\mbox{for all }x\in\Omega,\\
\left|\frac{\alpha\nabla w}{\|\nabla w\|_{p(x)}}\right|^{p_h(x)}\le \left(1+\frac{\alpha\nabla w}{\|\nabla w\|_{p(x)}}\right)^N\in L^1(\Omega)\quad\mbox{for all }h\in\mathbb N.
\end{gathered}
$$
Therefore, by the dominated convergence theorem, we obtain 
$$\lim_{h\to\infty}\varrho_{p_h(x)}\left(\frac{\alpha \nabla w}{\|\nabla w\|_{p(x)}}\right)=\varrho_{p(x)}\left(\frac{\alpha \nabla w}{\|\nabla w\|_{p(x)}}\right)\le \alpha \varrho_{p(x)}\left(\frac{\nabla w}{\|\nabla w\|_{p(x)}}\right)=\alpha<1$$
for all $\alpha\in(0,1)$.
Thus, for $h$ sufficiently large $\varrho_{p_h(x)}\left(\alpha \nabla w/\|\nabla w\|_{p(x)}\right)<1$,
which in turn gives 
$$
\Big\|\alpha\nabla w/\|\nabla w\|_{p(x)}\Big\|_{p_h(x)}<1
$$
by Proposition \ref{ubp}. Whence 
$$
\limsup_{h\to\infty}\|\nabla w\|_{p_h(x)}\le\frac{\|\nabla w\|_{p(x)}}\alpha\quad\mbox{for all }\alpha\in(0,1)
$$
and by the arbitrariness of $\alpha$ the assertion is proved. 
\end{proof}

\begin{theorem}\label{gsup}
Let $p,\,(p_h)\subset \mathscr C$ be such that $p_h\to p$ pointwise. Then 
\begin{equation}\label{glimsup}\mathscr E_{p(x)}(u)\ge\Big(\Gamma-\limsup_{h\to\infty}\mathscr E_{p_h(x)}\Big)(u)
\,\,\quad\mbox{for all }u\in L^1(\Omega).\end{equation}
\end{theorem}
\begin{proof} Suppose that $\mathscr E_{p(x)}(u)<\infty$ (otherwise \eqref{glimsup} is obvious) and take $b\in\mathbb R$ such that $b>\mathscr E_{p(x)}(u)$. Let $\delta>0$ and $w\in C^1_{\mathrm{c}}(\Omega)$ with $\|w-u\|_1<\delta$ and $\|\nabla w\|_{p(x)}<b$, then $\|\nabla w\|_{p_h(x)}\to\|\nabla w\|_{p(x)}$ by Lemma \ref{conv}. Therefore, 
$$
b>\lim_{h\to\infty}\mathscr E_{p_h(x)}(w),
$$
and in turn
$$b> \limsup_{h\to\infty}(\inf\{\mathscr E_{p_h(x)}(v)\,:\, \|v-u\|_1<\delta\}).$$
By the arbitrariness of $\delta>0$ we get 
$$
b\ge\Big(\Gamma-\limsup_{h\to\infty}\mathscr E_{p_h(x)}\Big)(u)
$$
and since $b>\mathscr E_{p(x)}(u)$ is arbitrary, we obtain \eqref{glimsup}.
\end{proof}

\begin{lemma}\label{embedding}
Let $p,q:\Omega\to [1,\infty)$ be measurable functions with $p(x)\le q(x)$ for a.a. $x\in\Omega$. Then $L^{q(x)}(\Omega)\hookrightarrow L^{p(x)}(\Omega)$ with embedding constant less or equal than
\begin{equation}\label{emconst}
C(|\Omega|,p,q):=\left[\left(\frac pq\right)_++\left(1-\frac pq\right)_+\right]\max\{|\Omega|^{(1/p-1/q)_+},|\Omega|^{(1/p-1/q)_-}\}.
\end{equation}
In particular, $C(|\Omega|,p,q_j)\to 1$ whenever $q_j$ converges uniformly to $p$.
\end{lemma}
\begin{proof} For all $u\in L^{q(x)}(\Omega)$, by H\"older's inequality  
\cite[cf.\ (3.2.23)]{Base}
$$\|u\|_{p(x)}\le\left[\left(\frac pr\right)_++\left(\frac pq\right)_+\right]\|1\|_{r(x)}\|u\|_{q(x)},$$
where $1/r:=1/p-1/q$ a.e. in $\Omega$. Moreover, by  
\cite[Lemma 3.2.5]{Base}, we get
$$
\|1\|_{r(x)}\le\max\{|\Omega|^{1/r_-},|\Omega|^{1/r_+}\},
$$
which concludes the proof.
\end{proof}

\begin{theorem}\label{ginf}
Let $p,\,(p_h)\subset \mathscr C$ be such that $p(x)\le p_h(x)$ for all $h\in\mathbb N$ and $x\in\Omega$, and $p_h\to p$ uniformly in $\Omega$. Then 
\begin{equation}\label{gliminf}
\mathscr E_{p(x)}(u)\le\Big(\Gamma-\liminf_{h\to\infty}\mathscr E_{p_h(x)}\Big)(u)\,\,
\quad\mbox{for all $u\in L^1(\Omega)$}.
\end{equation}
\end{theorem}
\begin{proof}
If $\left(\Gamma-\liminf_{h\to\infty}\mathscr E_{p_h(x)}\right)(u)=+\infty$ there is nothing to prove. In the other case, take $b\in\mathbb R$ such that $b>\left(\Gamma-\liminf_{h\to\infty}\mathscr E_{p_h(x)}\right)(u)$. By virtue of \cite[Proposition 8.1-(b)]{dalmaso1993} there exists a sequence $(u_h)\subset L^1(\Omega)$ such that $u_h\to u$ in $L^1(\Omega)$ and 
$$
\Big(\Gamma-\liminf_{h\to\infty}\mathscr E_{p_h(x)}\Big)(u)=\liminf_{h\to\infty}\mathscr E_{p_h(x)}(u_h).
$$
Hence, there is a subsequence $(p_{h_n})$ for which $$\sup_{n\in\mathbb N} \mathscr E_{p_{h_n}(x)}(u_{h_n})<b.$$
Let $(v_n)\subset C^1_{\mathrm{c}}(\Omega)$ verify 
$$\|v_n-u_{h_n}\|_1 <\frac1n,\quad\mathscr E_{p_{h_n}(x)}(v_n)<b\quad\mbox{for all }n\in\mathbb N.$$
Then $v_n\to u$ in $L^1(\Omega)$ and, by the embedding $W^{1,p_{h_n}(x)}_0(\Omega)\hookrightarrow W^{1,p(x)}_0(\Omega)$,
$$b>\|\nabla v_n\|_{p_{h_n}(x)}\ge \frac{\|\nabla v_n\|_{p(x)}}{C(|\Omega|,p,p_{h_n})}\quad\mbox{for all }n\in\mathbb N,$$
where $C(|\Omega|,p,p_{h_n})$ is given in \eqref{emconst} with $q=p_{h_n}$ and $C(|\Omega|,p,p_{h_n})\le 2(1+|\Omega|)<\infty$ for all $n$. Therefore, $(v_n)$ is bounded in the reflexive space $W^{1,p(x)}_0(\Omega)$ and so there exists a subsequence $(v_{n_j})$ such that $v_{n_j}\rightharpoonup u$ in $W^{1,p(x)}_0(\Omega)$. By Lemma \ref{embedding} and the uniform convergence of $p_{h_{n_j}}$ to $p$, 
$$
\lim_{j\to\infty}C(|\Omega|,p,p_{h_{n_j}})=1,
$$ 
and, together with the weak lower semicontinuity of the norm, we get
$$b\ge\liminf_{j\to\infty}\frac{\|\nabla v_{n_j}\|_{p(x)}}{C(|\Omega|,p,p_{h_{n_j}})}\ge\|\nabla u\|_{p(x)}=\mathscr E_{p(x)}(u).$$
In conclusion, by the arbitrariness of $b$, we obtain \eqref{gliminf}.
\end{proof}

\begin{lemma}\label{jlsc} Let $p,\,(p_h)\subset \mathscr C$ be such that $p_h\to p$ pointwise, $u\in L^{p(x)}(\Omega)$, $u_h\in L^{p_h(x)}(\Omega)$ for all $h$, and $u_h\to u$ a.e. in $\Omega$. Then 
$$\|u\|_{p(x)}\le \liminf_{h\to\infty}\|u_{h}\|_{p_h(x)}.$$ 
\end{lemma}
\begin{proof} Suppose that $\liminf_{h\to\infty}\|u_{h}\|_{p_h(x)}<\infty$ (otherwise there is nothing to prove) and take any $\alpha\in\mathbb R$ such that $\alpha>\liminf_{h\to\infty}\|u_{h}\|_{p_h(x)}$. Then there exists a subsequence $(p_{h_j})$ for which $\|u_{h_j}\|_{p_{h_j}(x)}< \alpha$ for all $j$. Hence, $\varrho_{p_{h_j}(x)}\left(u_{h_j}/\alpha\right)<1$ and by Fatou's Lemma
$$\int_\Omega\left|\frac{u}{\alpha}\right|^{p(x)}dx \le \liminf_{j\to\infty}\int_\Omega\left|\frac{u_{h_j}}{\alpha}\right|^{p_{h_j}(x)}dx\le1.$$
Thus, by Proposition \ref{ubp}, $\|u/\alpha\|_{p(x)}\le1$, that is $\|u\|_{p(x)}\le \alpha$. The conclusion follows by the arbitrariness of $\alpha$. 
\end{proof}

\begin{lemma}\label{le} Let $p,\,(p_h)\subset \mathscr C$ and $p_h\to p$ pointwise and suppose that for some $\bar h\in\mathbb N$  
\begin{equation}
\label{lbound}p_I:=\inf_{h\ge\bar h}(p_{h})_->1
\end{equation}
and
\begin{equation}
\label{ubound}
p_S:=\sup_{h\ge\bar h}(p_{h})_+<p_I^*,\quad\mbox{where }p_I^*=\frac{Np_I}{N-p_I}.
\end{equation} 
Then, for every sequence $(u_h)$ such that $u_h\in W^{1,p_h(x)}_0(\Omega)$ for all $h$ and $u_h\rightharpoonup u$ in $W^{1,p_I}_0(\Omega)$, there exists a subsequence $(u_{h_n})$ for which 
$$
\lim_{n\to\infty}\varrho_{p_{h_n}(x)}(u_{h_n})=\varrho_{p(x)}(u).
$$   
\end{lemma}
\begin{proof} By \eqref{lbound} and Lemma \ref{embedding}, $u_h\in W^{1,p_I}_0(\Omega)$ for all $h\ge\bar h$. Since $p_S<p_I^*$, $W^{1,p_I}_0(\Omega)$ is compactly embedded in $L^{p_S}(\Omega)$ and so $u_h\to u$ in $L^{p_S}(\Omega)$. 
Then there exists a subsequence $(u_{h_n})$ and a function $v\in L^{p_S}(\Omega)$ for which $u_{h_n}\to u$ and $|u_{h_n}|\le |v|$  a.e.\ in $\Omega$. Whence, a.e.,
\begin{align*}
\lim_{n\to\infty}|u_{h_n}|^{p_{h_n}(x)}&=|u|^{p(x)},\\
|u_{h_n}|^{p_{h_n}(x)}& \le 1+|v|^{p_S}\in L^1(\Omega) \quad\mbox{for all }n\in\mathbb N.
\end{align*}
In conclusion, by the dominated convergence theorem we obtain
$$
\lim_{n\to\infty}\int_\Omega|u_{h_n}|^{p_{h_n}(x)}dx=\int_\Omega|u|^{p(x)}dx,
$$
namely the assertion.
\end{proof}

\begin{remark}
Condition \eqref{lbound} is valid for instance when $p_h\to p$ uniformly or when $p_h\searrow p$ pointwise, while assumption \eqref{ubound} is a consequence of \eqref{lbound} when $p_I> N/2$, being $p_S\le N$.
\end{remark}

\begin{theorem}\label{ug}
Let $p,\,(p_h)\subset \mathscr C$, $p_h\to p$ pointwise and let \eqref{lbound}-\eqref{ubound} hold for some $\bar h \in\mathbb N$. Then, for every subsequence $(p_{h_n})$ and for every sequence $(u_n)\subset L^1(\Omega)$ verifying
\begin{equation}\label{sup}\sup_{n\in\mathbb N}\mathscr E_{p_{h_n}(x)}(u_n)<\infty,\end{equation}
there exists a subsequence $(u_{n_j})$ such that, as $j\to\infty$,
\begin{equation*}
\begin{gathered}u_{n_j}\to u\quad\mbox{in }L^1(\Omega),\\
g_{p_{h_{n_j}}(x)}(u_{n_j})\to g_{p(x)}(u).
\end{gathered}
\end{equation*} 
\end{theorem}
\begin{proof} For all $h_n\ge\bar h$, $W^{1,p_{h_n}(x)}_0(\Omega)\hookrightarrow W^{1,p_I}_0(\Omega)$ with embedding constant  less than or equal to $2(1+|\Omega|)$ (cf.\ \cite[Corollary 3.3.4]{Base}), then
$$
\|\nabla u_n\|_{p_I}\le 2(1+|\Omega|)\|\nabla u_n\|_{p_{h_n}(x)}\le 2b(1+|\Omega|),
$$
where $$b:=\sup_{n\in\mathbb N}\mathscr E_{p_{h_n}(x)}(u_n).$$ Since $W^{1,p_I}_0(\Omega)$ is reflexive, $(u_n)$ admits a subsequence $(u_{n_j})$ weakly convergent to $u$ in  $W^{1,p_I}_0(\Omega)$. Thus, $u_{n_j}\to u$ in $L^1(\Omega)$ and up to a subsequence $u_{n_j}\to u$ a.e. in $\Omega$.
For the second part of the statement, we have to prove that $\|u_{n_j}\|_{p_{h_{n_j}}(x)}\to\|u\|_{p(x)}$. By Lemma~\ref{jlsc} we know that 
$$
\|u\|_{p(x)}\le\liminf_{j\to\infty}\|u_{n_j}\|_{p_{h_{n_j}}(x)}.
$$
Now, for every real number 
$$
\alpha<\limsup_{j\to\infty}\|u_{n_j}\|_{p_{h_{n_j}}(x)},
$$ 
there exists a subsequence, still denoted by $(p_{h_{n_j}})$, for which $\alpha<\|u_{n_j}\|_{p_{h_{n_j}}(x)}$ for all $j$, and so 
$$1<\int_{\Omega}\left|\frac{u_{n_j}}\alpha\right|^{p_{h_{n_j}}(x)}dx,$$
by virtue of Proposition \ref{ubp}. Therefore, Lemma \ref{le} yields
$$1\le\lim_{j\to\infty}\varrho_{p_{h_{n_j}}(x)}\left(\frac{u_{n_j}}\alpha\right)=\varrho_{p(x)}\left(\frac{u}\alpha\right),$$
that is $\|u\|_{p(x)}\ge\alpha$ again by unit ball property. The conclusion follows by the arbitrariness of $\alpha$. 
\end{proof} 
 
 \noindent
We need to show that the minimax values with respect to the $W^{1,p(x)}_0(\Omega)$-topology are equal to those with respect to the weaker topology $L^1(\Omega)$.
To this aim, let $\mathcal W^{(m)}_{p(x)}$ be the family of those subsets $K$ of 
$$\{u\in W^{1,p(x)}_0(\Omega)\,:\,g_{p(x)}(u)=1 \}$$
which are compact and symmetric (i.e. $K=-K$), for which $i(K)\ge m$ with respect to the norm topology of $W^{1,p(x)}_0(\Omega)$, where $i$ denotes the Krasnosel'ski\u{\i} genus.
Furthermore, denote by $\mathcal K^{(m)}_{s,p(x)}$ the family of compact and symmetric subsets $K$ of 
$$\{u\in L^1(\Omega)\,:\, g_{p(x)}(u)=1\}$$
such that $i(K)\ge m$, with respect to the topology of $L^1(\Omega)$. 

\begin{theorem}\label{topol} 
Let $p\in\mathscr C$ and $\mathscr E_{p(x)}:L^1(\Omega)\to [0,\infty]$ be the function defined in \eqref{Epx}. Then, $\mathscr E_{p(x)}$ is convex, even and positively homogeneous of degree 1. Moreover, for every integer $m\ge1$, we have 
\begin{equation}\label{infsup}\inf_{K\in\mathcal K^{(m)}_{s,p(x)}}\sup_{K}\mathscr E_{p(x)}=\inf_{K\in\mathcal W^{(m)}_{s,p(x)}}\sup_{K}\mathscr E_{p(x)}.\end{equation}
\end{theorem}
\begin{proof} The fact that $\mathscr E_{p(x)}$ is convex, even and positively homogeneous of degree 1 follows easily by the definition. Furthermore, by Proposition \ref{p51}-(b) we know that for all $b\in\mathbb R$ the restriction of $g_{p(x)}$ to the set $\{v\in L^1(\Omega)\,:\, \mathscr E_{p(x)}(v)\le b\}$ is $L^1(\Omega)$-continuous. {\it A fortiori} the restriction of $g_{p(x)}$ to the same set is continuous with respect to the stronger topology $W^{1,p(x)}_0(\Omega)$ and the conclusion follows by \cite[Corollary 3.3]{DM14}.
\end{proof}

%
%

\section{Proof of Theorem~\ref{main}}

\noindent Due to Proposition \ref{p51} and the first part of Theorem \ref{topol}, the functionals $\mathscr E_{p(x)}$, $g_{p(x)}$, $\mathscr E_{p_h(x)}$ and $g_{p_h(x)}$ for all $h\in \mathbb N$ satisfy all the structural assumptions required in Section 4 of \cite{DM14}.   
Moreover, by Theorems \ref{gsup} and \ref{ginf}, we know that
$$
\mathscr E_{p(x)}(u)=\Big(\Gamma-\lim_{h\to\infty}\mathscr E_{p_h(x)}\Big)(u)\,\,\quad\mbox{for all $u\in L^1(\Omega)$}.
$$
Therefore, together with Theorem \ref{ug}, all the hypotheses of \cite[Corollary 4.4]{DM14} are verified and so
we can infer that
$$
\inf_{K\in\mathcal K^{(m)}_{s,p(x)}}\sup_{u\in K}\mathscr E_{p(x)}(u)=\lim_{h\to\infty}\Big(\inf_{K\in\mathcal K^{(m)}_{s,p_h(x)}}\sup_{u\in K}\mathscr E_{p_h(x)}(u)\Big).
$$ 
Finally, by \eqref{infsup} the last equality reads as
$$
\lambda_{p(x)}^{(m)}=\inf_{K\in\mathcal W^{(m)}_{s,p(x)}}\sup_{u\in K}\mathscr E_{p(x)}(u)=\lim_{h\to\infty}\Big(\inf_{K\in\mathcal W^{(m)}_{s,p_h(x)}}\sup_{u\in K}\mathscr E_{p_h(x)}(u)\Big)=\lim_{h\to\infty} \lambda_{p_h(x)}^{(m)}
$$ 
which proves the assertion. 

\begin{remark}\rm
Significant progresses were recently achieved in the framework of regularity theory for 
minimisers of a class of {\em double phase} integrands of the Calculus of Variations, see 
\cite{BarColMin1,BarColMin2,BarKuuMin,ColMing1,ColMing2} and the references therein. The model case  is
$$
u\mapsto \int_{\Omega} (|\nabla u|^p+a(x)|\nabla u|^q)dx,\quad\,\,\,
q>p,\,\,\,  a(\cdot)\geq 0,
$$
and it can be embedded into the class of Musielak-Orlicz spaces, with Orlicz norm
$$
\|u\|_{L^{\mathcal H}}=\inf\Big\{\lambda>0: \int_{\Omega} {\mathcal H}\Big(x,\frac{|u(x)|}{\lambda}\Big)dx\leq 1 \Big\},\quad
{\mathcal H}(x,s):=s^p+a(x)t^q,\,\,\,\, t\geq 0, \,\,\, x\in\Omega. 
$$
For a given topological index, such as the Krasnosel'ski\u{\i} genus or the $\mathbb{Z}_2$-cohomo\-logical index,
we plan to investigate in a forthcoming paper the asymptotic growth, the stability of the nonlinear eigenvalues $\lambda^{(m)}_{a,p,q}$, 
and the basic properties of the first eigenvalue.
\end{remark}

\bigskip
\bigskip

\end{document}